\documentclass[reqno]{amsart}

\usepackage{cite}
\usepackage{graphicx,subfigure,xcolor}

\numberwithin{equation}{section}

\usepackage[latin1]{inputenc}
\usepackage[english]{babel}

\usepackage{amsmath,amsthm,amsfonts,latexsym,amssymb,mathtools}
\usepackage[colorlinks]{hyperref}
\hypersetup{linkcolor=blue,citecolor=blue,filecolor=black,urlcolor=blue}
\usepackage{comment}

\usepackage{color}

{ \theoremstyle{plain}
\newtheorem{theorem}{Theorem}[section]
\newtheorem{proposition}[theorem]{Proposition}
\newtheorem{lemma}[theorem]{Lemma}
\newtheorem{corollary}[theorem]{Corollary}
  \theoremstyle{remark}

  \theoremstyle{definition}
\newtheorem{definition}[theorem]{Definition}
}

\def\R{\mathbb{R}}

\def\N{\mathbb{N}}

\begin{document}


\title[Supercritical problems in exterior domains]{Multiplicity and symmetry breaking for supercritical elliptic problems in exterior domains}

\author[A. Boscaggin]{Alberto Boscaggin}
\address{Alberto Boscaggin\newline\indent
Dipartimento di Matematica
\newline\indent
Universit\`a di Torino
\newline\indent
via Carlo Alberto 10, 10123 Torino, Italy}
\email{alberto.boscaggin@unito.it}

\author[F. Colasuonno]{Francesca Colasuonno}
\address{Francesca Colasuonno\newline\indent
Dipartimento di Matematica
\newline\indent
Universit\`a di Bologna
\newline\indent
p.zza di Porta San Donato 5, 40126 Bologna, Italy}
\email{francesca.colasuonno@unibo.it}

\author[B. Noris]{Benedetta Noris}
\address{Benedetta Noris \newline \indent 
Dipartimento di Matematica \newline\indent
Politecnico di Milano \newline\indent
p.zza Leonardo da Vinci 32, 20133 Milano, Italy}
\email{benedetta.noris@polimi.it}

\author[T. Weth]{Tobias Weth}
\address{Tobias Weth\newline \indent
Institut f\"ur Mathematik\newline \indent
Goethe-Universit\"at Frankfurt\newline \indent
Robert-Mayer-Str. 10, D-60629 Frankfurt am Main, Germany}
\email{weth@math.uni-frankfurt.de}


\begin{abstract} 
We deal with the following semilinear equation in exterior domains
\[  -\Delta u + u = a(x)|u|^{p-2}u,\qquad u\in H^1_0({A_R}), \]
where ${A_R} := \{x\in\mathbb{R}^N:\, |x|>{R}\}$, $N\ge 3$, $R>0$. Assuming that the weight $a$ is positive and satisfies some symmetry and monotonicity properties, we exhibit a positive solution having the same features as $a$, for values of $p>2$ in a suitable range that includes exponents greater than the standard Sobolev critical one.
In the special case of radial weight $a$, our existence result ensures multiplicity of nonradial solutions. We also provide an existence result for supercritical $p$ in nonradial exterior domains.
\end{abstract}

\maketitle

\section{Introduction}

In this paper we study existence and multiplicity of solutions to the following semilinear problem
\begin{equation}\label{P}
   \begin{cases}
  -\Delta u + u = a(x)u^{p-1} \qquad &\mbox{in }{A_R}\\
u>0 &\mbox{in }{A_R}\\
u\in H^1_0({A_R}) &
  \end{cases}
\end{equation}
where ${A_R}$ is the exterior of a ball of radius $R>0$, namely ${A_R} := \{x\in\mathbb{R}^N:\, |x|>{R}\}$, with $N\ge 3$, $a> 0$, and $p>2$ belongs to a suitable range that includes exponents greater than the Sobolev critical one. 

This problem lacks compactness due to both the unboundedness of the domain and the supercritical growth of the nonlinearity. We are able to recover compactness under the assumption that the weight $a$ satisfies some symmetry and monotonicity properties, working in a closed and convex subset of $H^1_0({A_R})$-functions having the same symmetry and monotonicity features as $a$. Our analysis covers the case of radial weights, in which we get, under suitable assumptions, multiple, rotationally nonequivalent, nonradial solutions. In particular, for the well-studied problem 
\begin{equation}\label{eq:pbOmega}
   \begin{cases}
  -\Delta u + u = u^{p-1} \qquad &\mbox{in }\Omega,\\
u>0 &\mbox{in }\Omega,\\
u\in H^1_0(\Omega),
  \end{cases}
\end{equation}
with constant weight $a\equiv 1$, we obtain the existence of nonradial solutions when $\Omega={A_R}$, $N \ge 3$, and $p \geq 2+\frac{8N}{(N-2)^2}$, see Corollary~\ref{coro:multiplicity} below.

We recall that, if $\Omega =\{x\in\R^N:\, |x|<R\}$, problem \eqref{eq:pbOmega} does not admit nontrivial solutions in the critical or supercritical regime, by the Poho\v{z}aev identity. Moreover, the Gidas-Ni-Nirenberg symmetry preservation result \cite{GNN} ensures that every regular solution is radial. On the other hand, when the domain is a bounded annulus, not only there are no natural constrains on $p$ for the existence of nontrivial solutions, but also symmetry breaking phenomena may arise, cf. the introduction of papers \cite{Li,B}. More precisely, for bounded annuli with sufficiently small interior radius, a symmetry preservation result is proved in \cite{GPY}, while, in expanding annuli, the existence and multiplicity of nonradial solutions is obtained in \cite{BrezisNirenberg,Coffman,Li,B,CW,gladiali-et-al2010,ByeonKimPistoia2013}. 
In the more recent paper \cite{BCNW2023}, we exhibit an explicit threshold of the inner radius above which the problem admits a nonradial solution, see also \cite{CowanMoameni2022JDE,CowanMoameni2022}. As remarked in \cite{BCNW2023}, this nonradial solution is axially symmetric but not foliated Schwarz symmetric and therefore has high Morse index, $i> N$, by \cite[Theorem 1.1]{PW}. The fact that the sufficient condition for symmetry breaking does not involve the outer radius motivated the investigation of similar results in exterior domains, heuristically letting the outer radius go to infinity. From a technical point of view, recovering the necessary compactness is a more delicate issue in exterior domains, due to the unboundedness; nonetheless, the threshold obtained for symmetry breaking is the same as in bounded annuli.

In order to state our main results, we shall first introduce some notation. Let $\R^N=\R^m\times\R^{N-m}$ with $m \in \{2,\ldots,N-1\}$. We shall work with functions that only depend on $r=|x|=\sqrt{x_1^2+\ldots+x_N^2}\in [R,\infty)$ and $\theta \in \left[0,\pi/2\right]$ given by
\begin{equation}\label{eq:theta}
\theta= \arctan\left(\frac{\sqrt{x_{m+1}^2+\ldots+x_N^2}}{\sqrt{x_1^2+\ldots+x_m^2}}\right) = 
\arcsin\left(\frac{1}{r}\sqrt{x_{m+1}^2+\ldots+x_N^2}\right) ,
\end{equation}
namely
\[
u(x) = \mathfrak{u}\left(r,\theta \right)\quad \text{for a function } \mathfrak{u}: [R,\infty)\times\left[0,\frac{\pi}{2}\right] \to \mathbb{R}.
\]
We observe that the symmetry assumption $u(x)=\mathfrak u(r,\theta)$ implies that $u$ is even with respect to every variable $x_i$. 
When $m=N-1$,  we get as a particular case axially symmetric functions  with respect to the $x_N$-axis, which are symmetric with respect to the hyperplane $x_N=0$ (see \cite{BCNW2023}).

In addition to these symmetry assumptions, we impose the positivity and monotonicity properties that we collect in the following class of functions
\begin{equation}
  \label{def-cal-C}
\mathcal C:= \Bigl\{u\in W^{1,1}_{\mathrm{loc}}({A_R})\,:\, u=\mathfrak{u}(r,\theta),\: \,u\ge 0,\: \mathfrak{u}_\theta\le 0\:  \mbox{ a.e. in } {A_R}
\Bigr\},
\end{equation}
where $\mathfrak{u}_\theta$ stands for the weak partial derivative of $u$ with respect to the variable $\theta$.
Notice that, being $u\in W^{1,1}_{\mathrm{loc}}({A_R})$, we have $\mathfrak{u}\in W^{1,1}_{\mathrm{loc}}((R,\infty)\times(0,\pi/2))$ (see for example \cite[Proposition 9.6]{brezis2010}), so that $\mathfrak{u}_\theta \in L^1_{\mathrm{loc}}((R,\infty)\times(0,\pi/2))$. 
Furthermore, according to \cite[\S 4.9.2, Theorem 2]{EG}, there exists a representative of $\mathfrak u$, still denoted by $\mathfrak u$, such that for almost every $r\in (R,\infty)$, $\mathfrak u(r,\cdot)$ is absolutely continuous. As a consequence, the sign assumption on $\mathfrak u_\theta$ appearing in the definition of $\mathcal C$ implies that for a.e. $r\in(R,\infty)$,
\begin{equation}\label{eq:monotonicity}
\mathfrak{u}(r,\theta_1)\ge \mathfrak{u}(r,\theta_2)\quad \mbox{ if } \ 0\leq \theta_1<\theta_2< \pi/2.
\end{equation}
In particular, $\mathfrak{u}$ achieves its maximum on the set where $\theta=0$, that is the linear subspace $\{x\in\mathbb R^N\,:\, x_{m+1}=\dots=x_N=0\}$.
For future use, we also observe that the set $\mathcal C$ is a convex cone.

Concerning the weight function $a$ appearing in \eqref{P}, we assume that
\begin{equation}\label{a_assumptions}
a \in \mathcal C \cap L^\infty({A_R}), \quad |\nabla a|\in L^2({A_R}), \quad \text{and}\quad a \not \equiv 0.
\end{equation}
In particular, $a=\mathfrak{a}(r,\theta)$.
Under these hypotheses, we look for solutions to \eqref{P} belonging to the following subset of $\mathcal C$
\begin{equation}\label{cone}
\mathcal K:= \mathcal C \cap H^1_0({A_R}).
\end{equation}
$\mathcal K$ is a convex cone, which is closed with respect to the $H^1$-norm, see Lemma~\ref{le:Kcone} below.
Moreover, under the assumption
\[
p\in(2,2^*_{N-m+1}),
\]
$\mathcal K$ is continuously embedded in $L^p({A_R})$,  see Lemma \ref{le:continuous_embedding} ahead.
Here, for every integer $n \geq 2$, $2^*_n$ denotes the critical Sobolev exponent in $\R^n$, that is
\begin{equation}\label{eq:critical_exponent}
2^*_{n}=
\begin{cases}
\frac{2n}{n-2} \qquad &\text{if } n\geq 3 \\
\infty & \text{if } n = 2.
\end{cases}
\end{equation}
We also adopt the standard notation $2^*:=2^*_N$ for the critical Sobolev exponent in dimension $N$.
We remark that  
\begin{equation}\label{eq:assumptions_p}
2^*_{N-m+1}>2^*
\end{equation}
holds for every  $N\geq3$ and $m \in \{2,\ldots,N-1\}$. 

As a consequence of the embedding of $\mathcal K$ in $L^p({A_R})$, the $C^2$-functional $I: H^1_0({A_R}) \cap L^p({A_R}) \to \mathbb R$ given by
\begin{equation}\label{eq:I_def}
I(u):=\frac{1}{2}\int_{A_R}(|\nabla u|^2+u^2)dx-\frac{1}{p}\int_{A_R} a|u|^p dx
\end{equation}
is well-defined in $\mathcal K$. This allows us to look for a nontrivial solution $u\in \mathcal K$ of \eqref{P} by variational methods. 
To this aim, we introduce the following Nehari-type set
\[
\mathcal N_\mathcal K:= \{u \in \mathcal{K}\::\: u \not\equiv 0, \,I'(u)u=0\}
\]
and the Nehari value
\begin{equation}\label{eq:Ic-equivalent}
c_I := \inf_{u \in {\mathcal N}_{\mathcal{K}}} I(u).
\end{equation}

With this notation, our first main result reads as follows.

\begin{theorem}\label{thm:main_existence}
Let $N\geq3$, $R>0$, let $a$ satisfy \eqref{a_assumptions}, and let $p\in(2,2^*_{N-m+1})$. Then there exists a nontrivial solution $u$ of problem \eqref{P} such that 
\begin{equation}\label{eq:K-ground-state}
u\in\mathcal N_\mathcal K\quad\mbox{and}\quad I(u)=c_I>0.
\end{equation}
\end{theorem}

In the following, we shall refer to a solution of \eqref{P} satisfying \eqref{eq:K-ground-state} as a {\em $\mathcal{K}$-ground state solution}.

From \eqref{eq:assumptions_p}, we see that Theorem \ref{thm:main_existence} provides the existence of solutions to \eqref{P} also in the supercritical regime.

The proof of the previous theorem relies on a recent abstract result by Cowan and Moameni, \cite[Theorem 2.1]{CowanMoameni2022}, that we recall in Section \ref{subsec:abstract_theorem}. This result is based in turn on the critical point theory of Szulkin for non-smooth functionals \cite{S} and 
provides a general strategy to prove existence of solutions to supercritical problems under symmetry and monotonicity assumptions. Roughly speaking, it requires the presence of an invariant closed and convex subset of the functional space, in which compactness is recovered: such a set is typically defined in terms of symmetry and monotonicity properties. Similar ideas were previously exploited in \cite{BNW,ST,CM} to build solutions of supercritical Neumann problems in bounded domains, see also \cite{ABF-ESAIM,ABF-PRE,CN} for related results.

In the present paper, we prove that the abstract result in \cite{CowanMoameni2022} applies to problem \eqref{P} which has compactness issues due both to the supercritical growth of the nonlinearity and to the unboundedness of the domain. The 
crucial property to be proved is that the convex cone $\mathcal K$ is compactly embedded in $L^p({A_R})$. 
Such a property is proved in Lemma \ref{le:continuous_embedding}, and in Propositions \ref{le:subcritical_compact_embedding} and \ref{prop:compact_embedding} by a rather delicate argument. The main new idea here is to combine the monotonicity property of functions in $\mathcal K$ with a compact embedding result due to Willem \cite{Willem} applied in suitably chosen subdomains of ${A_R}$.

In case the weight $a$ is a radial function, that is
\begin{equation}\label{P-rad}
   \begin{cases}
  -\Delta u + u = a(|x|) u^{p-1} \qquad &\mbox{in }{A_R},\\
u>0 &\mbox{in }{A_R},\\
u\in H^1_0({A_R}),
  \end{cases} 
\end{equation}
the whole problem has radial symmetry, and so the existence of a nontrivial radial solution for every $p\in (2,\infty)$ is a consequence of the Strauss lemma, cf. \cite{Strauss}.
Furthermore, if $p\in (2,2^*_{N-m+1})$, Theorem~\ref{thm:main_existence} provides a $\mathcal{K}$-ground state solution of \eqref{P-rad} for every $m \in \{2,\ldots,N-1\}$.
A priori it is not clear whether these solutions are radial or not and whether solutions obtained for different choices of $m$ coincide or not.
It is easy to prove that solutions given by different values of $m$ are distinct unless they are radial, see Lemma \ref{lem:m_m'}; moreover, in Proposition \ref{prop:non_radial} we show that a $\mathcal K$-ground state solution of \eqref{P-rad} is nonradial if $p$ belongs to the following range
\begin{equation}\label{eq:suff_p}
2+ \frac{2N}{\bigl(\frac{N-2}{2}\bigr)^2+ R^2} \le p < 2^*_{N-m+1}.
\end{equation}
We remark that such a sufficient condition coincides with the one found for the same problem on a bounded annulus, cf. \cite{BCNW2023}. 
Condition \eqref{eq:suff_p} can equivalently be read as an assumption on $R$, when $p$ is fixed. In particular, inverting the first inequality in \eqref{eq:suff_p}, we define
\begin{equation}\label{eq:R*}
R^*:=
\begin{cases}
\left[ \frac{2N}{p-2}-\left(\frac{N-2}{2}\right)^2 \right]^{1/2}
\quad & \text{if } p<2+\frac{8N}{(N-2)^2} \\
0 \quad & \text{otherwise}.
\end{cases}
\end{equation}
Recalling the definition of $2^*_n$ in \eqref{eq:critical_exponent}, our main multiplicity and symmetry breaking result for problem \eqref{P-rad} can be stated as follows.

\begin{theorem}\label{thm:main-a-rad}
Let $N\geq3$ and let $a$ be a radial function satisfying \eqref{a_assumptions}.
For every integer $2\leq n \leq N-1$, if $2<p<2^*_n$ and $R>R^*$,  problem \eqref{P-rad} has $n$ rotationally nonequivalent solutions, one of which is radial and the others $n-1$ are not. The same result holds also for $R=R^*$ if $R^*>0$. 
\end{theorem}
We remark that, similarly as in \cite{BCNW2023}, it can be seen that the nonradial solutions obtained in the previous theorem are not foliated Schwarz smmetric, so they have high Morse index, $i> N$, by \cite[Theorem 1.3]{gladiali-pacella-weth-2010}.

As a particular case of Theorem~\ref{thm:main-a-rad},  we exhibit conditions on $p$ and $n$ that ensure the existence of $n$ distinct solutions of \eqref{P-rad} on any unbounded annulus of inner radius $R>0$.
In view of \eqref{eq:R*},  this amounts to ensure that $R^*=0$ and that the interval 
\[
\left(2+\frac{8N}{(N-2)^2} , 2^*_n \right)
\]
is non-empty.

\begin{corollary}\label{coro:multiplicity}
Let $R>0$, and let $a$ be a radial function satisfying \eqref{a_assumptions}.
\begin{itemize}
\item[(i)] Let $N\geq 3$. For
\begin{equation}\label{eq:coro1}
p\geq 2+\frac{8N}{(N-2)^2} 
\end{equation}
problem \eqref{P-rad} has $2$ distinct solutions, one of which is radial and the other is not.
\item[(ii)] Let $N\geq 6$ and let $n\geq3$ be an integer.  For 
\begin{equation}\label{eq:coro2}
n\leq \frac{N}{2} \quad\text{and}\quad
2+\frac{8N}{(N-2)^2} \leq p <2^*_n
\end{equation}
problem \eqref{P-rad} has $n$ rotationally nonequivalent solutions,  one of which is radial and the others are not.
\end{itemize}
\end{corollary}

Finally, we wish to add an existence result on more general nonradial exterior domains which are still symmetric with respect to rotations in $O(m) \times O(N-m)$. More precisely, for given $m \in \{2,\dots,N-1\}$ and a $C^2$-function $g:[0,\infty) \to \R$ with 
\begin{equation}
  \label{eq:assumption-g}
g(0)< 0 \qquad \text{and}\qquad 0 < \inf g' \le \sup g' \le1,
\end{equation}
we now consider the exterior domain
\begin{equation}
  \label{eq:def-A_g}
A_g:= \{x \in \R^N\::\: x_1^2+ \dots + x_m^2 + g(x_{m+1}^2 + \cdots + x_N^2) >0\}.
\end{equation}
The class of domains given by (\ref{eq:def-A_g}) can be seen as an unbounded variant of the class of {\em bounded domains of double revolution with monotonicity} as introduced in \cite{CowanMoameni2022JDE}. In particular, we note that, in the special case $g(\tau)=\kappa \tau - c$ with $0< \kappa \le 1$ and $c >0$, the set $A_g$ is the complement of an $O(m) \times O(N-m)$-symmetric ellipsoid in $\R^N$.   

In the following, we let $\mathcal C_g$ and $\mathcal K_g$ be given as in (\ref{def-cal-C}) and (\ref{cone}) with $A$ replaced by $A_g$. Moreover, we define $\mathcal N_{\mathcal K_g}:= \{u \in \mathcal{K}_g\::\: u \not\equiv 0, \,I'(u)u=0\}$. We then have the following existence result in nonradial exterior domains.

\begin{theorem}\label{thm:main_existence-nonradial}
  Let $N\geq3$, let $g \in C^2([0,\infty))$ satisfy (\ref{eq:assumption-g}), let $a$ satisfy \eqref{a_assumptions} with $A$ replaced by $A_g$, and let $p\in(2,2^*_{N-m+1})$. Then there exists a nontrivial solution $u$ of the problem
\begin{equation}\label{P_g}
   \begin{cases}
  -\Delta u + u = a(x)u^{p-1} \qquad &\mbox{in }A_g\\
u>0 &\mbox{in }A_g\\
u\in H^1_0(A_g) &
  \end{cases}
\end{equation}
such that 
\begin{equation}\label{eq:K-ground-state-g}
u\in\mathcal N_{\mathcal K_g}\quad\mbox{and}\quad I(u)=c_I>0.
\end{equation}
\end{theorem}

The paper is organized as follows. In Section \ref{sec-existence} we prove the compactness and the pointwise invariance property needed to apply the abstract theorem of \cite{CowanMoameni2022}; the section ends with the proof of the existence of a $\mathcal K$-ground state solution of problem \eqref{P} as stated in Theorem \ref{thm:main_existence}. Section \ref{sec:case-a-radial} is devoted to the radial problem \eqref{P-rad} and contains the proof of Theorem \ref{thm:main-a-rad} that guarantees the existence of multiple nonradial, rotationally nonequivalent solutions. Finally, in Section~\ref{sec:exist-nonr-exter} we sketch the proof of Theorem~\ref{thm:main_existence-nonradial}, which follows by very similar arguments as Theorem \ref{thm:main_existence} with rather straightforward modifications. 

In order to keep the notation simple, in Sections \ref{sec-existence} and \ref{sec:case-a-radial} we shall write $A$ in place of $A_R$, so we fix $R>0$ and set 
$$
{A} := \{x\in\mathbb{R}^N:\, |x|>{R}\}.
$$

\section{Proof of Theorem \ref{thm:main_existence}}\label{sec-existence}

\subsection{Abstract result}\label{subsec:abstract_theorem}

As already mentioned in the Introduction, our existence result relies on \cite[Theorem 2.1]{CowanMoameni2022}, which we report here in a version that has been sightly adapted to our case.

\begin{definition}\label{def:convex-cone}
Let $V$ be a linear space.
We say that $K\subset V$ is a convex cone if the following properties hold 
\begin{itemize}
\item[(i)] if $u\in K$ and $\lambda\ge0$, then $\lambda u \in K$;
\item[(ii)] if $u,\, -u\in K$, then $u=0$;
\item[(iii)] if $u,\, v\in K$, then $u+v\in K$.
\end{itemize}
\end{definition}

\begin{theorem}[{cf. \cite[Theorem 2.1]{CowanMoameni2022}}]\label{thm:thm2.1CM}
Let $\Omega$ be a domain in $\mathbb R^N$, $p>2$, and let $w \in L^\infty(\Omega)$ be a nonnegative function. Consider the problem 
\begin{equation}\label{eq:P-CM}
\begin{cases}
-\Delta u+u=w(x)|u|^{p-2}u\quad&\mbox{in }\Omega,\\
u\in H^1_0(\Omega),
\end{cases}
\end{equation}
and its formal Euler-Lagrange functional
\begin{equation}\label{eq:J_def}
J(u)=\frac{1}{2}\int_\Omega(|\nabla u|^2+u^2)dx-\frac{1}{p}\int_\Omega w|u|^p dx.
\end{equation}
Let $K\subset H^1_0(\Omega)$ be a closed convex cone. Suppose that the following three assumptions hold: 
\begin{itemize}
\item[(i)] \emph{Compactness:} $K \subset L^p(\Omega)$, and this embedding is compact. In other words: If $(u_n)_n$ is a bounded sequence in $K$ (w.r.t. the $H^1$-norm), then $(u_n)_n$ has a convergent subsequence in $L^p(\Omega)$.
\item[(ii)] \emph{Pointwise invariance property:} For each $u\in K$ there exists $v\in K$ that is a weak solution of the equation
  \[-\Delta v + v=w(x)|u|^{p-2}u.\]
\item[(iii)] There exists $e \in K \setminus \{0\}$ with $J(e) \le 0$.  
\end{itemize}
Then there exists a nontrivial weak solution $u\in K$ of \eqref{eq:P-CM} such that 
\[
J(u)=\inf_{\gamma\in\Gamma}\sup_{t\in[0,1]}J(\gamma(t))>0,
\]
with $\Gamma:=\{\gamma\in C([0,1]; K)\,:\gamma(0)=0	\neq\gamma(1),\,J(\gamma(1))\le 0\}$.
\end{theorem}


The variational characterization given in \cite{CowanMoameni2022} is expressed in a slightly different way, but coincides with the one given above. We remark that in our statement the set $K$ is required to be a cone, since we believe that this assumption is needed for the proof of the Palais-Smale condition. We provide such proof in Appendix \ref{sec:appendix}.

We shall apply the previous abstract theorem with $\Omega=A$, $w(x)=a(x)$, and $K=\mathcal{K}$. In the remaining part of this section we prove that the assumptions of Theorem \ref{thm:thm2.1CM} are satisfied in our case and conclude the proof of Theorem \ref{thm:main_existence}.

\subsection{Some useful change of variable formulas}\label{subsec:chofvar}

Throughout the paper, we shall deal with functions of the type
$u(x) = \mathfrak{u}\left(r,\theta \right)$ with $r = \vert x \vert$ and $\theta$ as in \eqref{eq:theta}. 
Notice that these functions can also be written in terms of the cartesian coordinates in $\mathbb R^2$ 
\begin{equation}\label{eq:s-t}
s:=r\cos\theta=\sqrt{x_1^2+\ldots+x_m^2}, \quad
t:=r\sin\theta=\sqrt{x_{m+1}^2+\ldots+x_N^2}, 
\end{equation}
with $(s,t) \in Q:= \{(s,t)\in\R^2:\, s^2+t^2>R^2, \, s\geq0, \, t\geq0\}$. 
Hence, in the following we shall write
\[u(x)=\mathfrak{u}(r,\theta)=\tilde{\mathfrak{u}}(s,t).\]
For the reader's convenience, we collect below some useful formulas which will be often used throughout the paper.

At first we observe that, by a change of variables in the integral,
\begin{equation}\label{eq:changeintegral}
\begin{aligned}
\int_A u(x)\, dx & = \omega_{N,m} \int_Q \tilde{\mathfrak{u}}(s,t) s^{m-1} t^{N-m-1}\,ds \,dt \\
& = \omega_{N,m} \int_{R}^\infty \int_0^{\frac{\pi}{2}} \mathfrak{u}(r,\theta) \mu(\theta) r^{N-1} \,dr\,d\theta,
\end{aligned}
\end{equation}
where $\omega_{N,m}:=\omega_{m-1}\omega_{N-m-1}$, with $\omega_k$ denoting the surface measure of the sphere $\mathbb{S}^k \subseteq \R^{k+1}$, with the convention that $\omega_0=2$, and 
\begin{equation}\label{eq:mu}
\mu(\theta):=(\cos\theta)^{m-1}(\sin\theta)^{N-m-1}.
\end{equation}

Second, we provide some formal expressions for the derivatives. To this end, we notice that 
\begin{equation}\label{eq:gradient-s-t}
u_{x_i}=
\begin{cases}
\displaystyle \frac{x_i}{s} \tilde{\mathfrak{u}}_s & \qquad i=1,\ldots,m \vspace{0.2cm}\\
\displaystyle \frac{x_i}{t} \tilde{\mathfrak{u}}_t & \qquad i=m+1,\ldots,N
\end{cases}
\end{equation}
and that, by passing to polar coordinates,
\[
\tilde{\mathfrak{u}}_s=\cos\theta\, \mathfrak{u}_r - \frac{\sin\theta}{r}\mathfrak{u}_\theta, \qquad 
\tilde{\mathfrak{u}}_t=\sin\theta\, \mathfrak{u}_r + \frac{\cos\theta}{r}\mathfrak{u}_\theta.
\]
Then,
\begin{equation}\label{eq:gradient-r-theta}
\vert \nabla u \vert^2 = \tilde{\mathfrak{u}}_s^2 + \tilde{\mathfrak{u}}_t^2 = \mathfrak{u}_r^2 + \frac{\mathfrak{u}_\theta^2}{r^2}.
\end{equation}
Moreover, we have
\[
u_{x_ix_i}=
\begin{cases}
\displaystyle \frac{x_i^2}{s^2}\tilde{\mathfrak{u}}_{ss}+\frac{1}{s}\left(1-\frac{x_i^2}{s^2}\right)\tilde{\mathfrak{u}}_s & \qquad i=1,\ldots,m \vspace{0.2cm}\\
\displaystyle \frac{x_i^2}{t^2}\tilde{\mathfrak{u}}_{tt}+\frac{1}{t}\left(1-\frac{x_i^2}{t^2}\right)\tilde{\mathfrak{u}}_t & \qquad i=m+1,\ldots,N.
\end{cases}
\]
Using the well known formula for the Laplacian in polar coordinates, namely
\[
\tilde{\mathfrak{u}}_{ss}+\tilde{\mathfrak{u}}_{tt}=\mathfrak{u}_{rr}+\frac{1}{r}\mathfrak{u}_r+\frac{1}{r^2}\mathfrak{u}_{\theta\theta},
\]
we finally find
\begin{equation}\label{eq:Deltau_s-r-th}
\begin{aligned}
\Delta u&=\tilde{\mathfrak{u}}_{ss}+\tilde{\mathfrak{u}}_{tt}+\frac{m-1}{s}\tilde{\mathfrak{u}}_s+\frac{N-m-1}{t}\tilde{\mathfrak{u}}_t
\\& = \mathfrak{u}_{rr}+ \frac{N-1}{r}\mathfrak{u}_r-\frac{(m-1)\tan\theta - (N-m-1)\cot\theta}{r^2}\mathfrak{u}_\theta +\frac{1}{r^2} \mathfrak{u}_{\theta\theta}.
\end{aligned}
\end{equation}

\subsection{Closedness of $\mathcal K$ and compactness}

We prove that the convex cone $\mathcal{K}$ defined in \eqref{cone} is closed. 
\begin{lemma}\label{le:Kcone}
The convex cone $\mathcal K$ is closed with respect to the $H^1(A)$-norm; as a consequence, it is weakly closed. 
\end{lemma}
\begin{proof}
Let $(u_n)_{n}\subset\mathcal K$ and $u\in H^1_0(A)$ be such that $u_n\to u$ in $H^1(A)$ as $n\to\infty$. 
By pointwise almost everywhere convergence up to a subsequence,  $u= \mathfrak{u}\left(r,\theta \right)$ is nonnegative and even with respect to $\theta$. 
Let us check that $\mathfrak{u}_\theta\le 0$. By \cite[Proposition 9.6]{brezis2010}, we can write
\[
0\geq \frac{\partial\mathfrak{u}_n}{\partial\theta} = \nabla u_n \cdot \frac{\partial x}{\partial \theta} \to \nabla u \cdot \frac{\partial x}{\partial \theta} =\mathfrak{u}_\theta
\]
almost everywhere, as $n\to\infty$. 
Then, $u \in \mathcal K$, proving that $\mathcal K$ is closed in the strong $H^1$-topology. Being $\mathcal K$ convex, we conclude that $\mathcal K$ is weakly closed, as well.
\end{proof}

We now prove the fundamental property that $\mathcal K$ is continuously embedded in $L^q(A)$, for exponents smaller than the Sobolev critical one in dimension $N-m+1$.

\begin{lemma}\label{le:continuous_embedding}
Let $m\in \{2,\ldots,N-2\}$, then for every $q\in [2,2^*_{N-m+1}]$ there exists 
a positive constant $C(q)$ such that 
\begin{equation}\label{eq:est3}
\|u\|_{L^q(A)}\leq C(q) \|u\|_{H^1(A)} \quad\mbox{for every }u\in \mathcal K.
\end{equation}
If $m=N-1$ then \eqref{eq:est3} holds for every $q\in [2,\infty)$.
\end{lemma}
\begin{proof}
Let $u\in \mathcal K$. By formula \eqref{eq:changeintegral} we have
\begin{equation}\label{eq:integral-change-var}
\begin{aligned}
\|u\|_{L^q(A)}^q & =
\omega_{N,m} \int_{R}^\infty \int_0^{\frac{\pi}{3}} |\mathfrak{u}(r,\theta)|^q \mu(\theta) r^{N-1} \,dr\,d\theta \\
& \hspace{0.7cm}+ \omega_{N,m} \int_{R}^\infty \int_{\frac{\pi}{3}}^{\frac{\pi}{2}} |\mathfrak{u}(r,\theta)|^q \mu(\theta) r^{N-1} \,dr\,d\theta \\
&=:I_{[0,\frac{\pi}{3}]}+I_{[\frac{\pi}{3},\frac{\pi}{2}]},
\end{aligned}
\end{equation}
where we have split the integral in the two regions $[R,\infty)\times[0,\pi/3]$ and $[R,\infty)\times[\pi/3,\pi/2]$.

Proceeding similarly to the proof of \cite[Theorem 3.1]{CowanMoameni2022JDE}, we shall show that there exists a constant $C_1>0$ such that 
\begin{equation}\label{eq:I_pi_3}
I_{[\frac{\pi}{3},\frac{\pi}{2}]} \leq C_1 I_{[0,\frac{\pi}{3}]}.
\end{equation}
Indeed, let $C_1$ be such that $(\sin\theta)^{N-m-1} \leq C_1 (\sin(\theta-\pi/4))^{N-m-1}$ for every $\theta \in (\pi/3,\pi/2)$. Being both $\mathfrak{u}$ and $\cos\theta$ decreasing in this interval, recalling the definition of $\mu(\theta)=(\cos\theta)^{m-1}(\sin\theta)^{N-m-1}$, we have
\begin{equation*}
\begin{aligned}
I_{[\frac{\pi}{3},\frac{\pi}{2}]} & \leq 
C_1 \, \omega_{N,m} \int_{R}^\infty \int_{\frac{\pi}{3}}^{\frac{\pi}{2}} \left|\mathfrak{u}\left(r,\theta-\frac{\pi}{4}\right)\right|^q \left(\cos\left(\theta-\frac{\pi}{4}\right)\right)^{m-1}\cdot \\
& \hspace{3,7cm}\cdot \left(\sin\left(\theta-\frac{\pi}{4}\right)\right)^{N-m-1} r^{N-1} \,dr\,d\theta \\
& \leq C_1 \, \omega_{N,m} \int_{R}^\infty \int_{\frac{\pi}{12}}^{\frac{\pi}{4}} \left|\mathfrak{u}\left(r,\theta\right)\right|^q \mu(\theta) r^{N-1 } \,dr\,d\theta
\leq C_1 \, I_{[0,\frac{\pi}{3}]},
\end{aligned}
\end{equation*}
so that \eqref{eq:I_pi_3} holds.

Now, in order to estimate $I_{[0,\frac{\pi}{3}]}$,  we let
\[
\tilde{u}(s,x_{m+1},\ldots,x_N)=u(x)
\]
for all $(s,x_{m+1},\ldots,x_N)\in\R^{N-m+1}$ such that $s^2+x_{m+1}^2+\ldots+x_N^2>R^2$, $s\geq0$. 
We have
\[
I_{[0,\frac{\pi}{3}]} =
\omega_{m-1} \int_\Omega |\tilde{u}(s,x_{m+1},\ldots,x_N)|^q s^{m-1}\,ds\,dx_{m+1}\ldots dx_N,
\]
with
\begin{multline*}
\Omega=\Big\{(s,x_{m+1},\ldots,x_N)\in \R^{N-m+1}: \, s^2+x_{m+1}^2+\ldots+x_N^2>R^2, s\geq0, \\
\arctan\left(\frac{1}{s}\sqrt{x_{m+1}^2+\ldots+x_N^2}\right)\in \left(0,\frac{\pi}{3}\right)\Big\}.
\end{multline*}
We notice that
\begin{equation}\label{eq:s_bound_below}
(s,x_{m+1},\ldots,x_N)\in\Omega \quad\text{implies}\quad s>\frac{R}{2}.
\end{equation}
We consider the rescaled function
\begin{equation}\label{eq:v_tilde_def}
\tilde{v}(s,x_{m+1},\ldots,x_N):=s^{\frac{m-1}{q}}\tilde{u}(s,x_{m+1},\ldots,x_N),
\end{equation}
so that
\[
I_{[0,\frac{\pi}{3}]} =
\omega_{m-1} \int_\Omega |\tilde{v}(s,x_{m+1},\ldots,x_N)|^q \,ds\,dx_{m+1}\ldots dx_N.
\]
Because of the assumptions on $q$, we are now in a position to apply the Sobolev embedding theorem in $\R^{N-m+1}$:
\[
I_{[0,\frac{\pi}{3}]} \leq 
\omega_{m-1} C_{S,q}^q \left( \int_\Omega \left( |\nabla \tilde{v}|^2+ \tilde{v}^2 \right) \,ds\,dx_{m+1}\ldots dx_N\right)^{q/2},
\]
where $C_{S,q}$ denotes the constant of the Sobolev embedding $L^q(\Omega)\subset H^1_0(\Omega)$ in $\R^{N-m+1}$. Rescaling back to the function $\tilde{u}$ we obtain, by \eqref{eq:v_tilde_def},
\[
I_{[0,\frac{\pi}{3}]} \leq C_2 
\left( \int_\Omega \left( s^{\alpha_1} |\nabla \tilde{u}|^2+ (s^{\alpha_1}+s^{\alpha_2}) \tilde{u}^2 \right) s^{m-1} \,ds\,dx_{m+1}\ldots dx_N\right)^{q/2},
\]
for a constant $C_2>0$ depending only on $q,N,m$, and with $\alpha_1=2(m-1)/q-m+1$ and $\alpha_2=2(m-1)/q-m-1$. As $\alpha_2<\alpha_1<0$ and $s$ is bounded from below in $\Omega$ by \eqref{eq:s_bound_below}, we can conclude that
\[
I_{[0,\frac{\pi}{3}]} \leq C_3 
\left( \int_\Omega \left( |\nabla \tilde{u}|^2+ \tilde{u}^2 \right) s^{m-1} \,ds\,dx_{m+1}\ldots dx_N\right)^{q/2}
\leq C_4 \|u\|_{H^1(A)}^q,
\]
where we used that $|\nabla \tilde{u}|^2=|\nabla u|^2$, by \eqref{eq:gradient-s-t}.
The statement follows by combining the last inequality with \eqref{eq:I_pi_3}.
\end{proof}

We remark that the continuous embedding proved in the previous lemma holds also with $m=1$, by the classical Sobolev embedding $H^1_0(A)\hookrightarrow L^{2^*}(A)$. The restriction $m\ge2$, that we require in the whole paper, is 
needed to ensure compactness, see Proposition \ref{prop:compact_embedding} below. 
As the proof of this proposition is quite involved, we need some preliminary notation and results.  

Let $k\geq 2$ be an even integer, we introduce the following subset of $A$
\begin{equation}\label{eq:Ak}
A_k:=\left\{x\in\mathbb{R}^N:\, |x|> R, \,  \theta\in \left[0,\left(1-\frac{1}{k}\right)\frac{\pi}{2}\right) \right\}
\end{equation}
with $\theta$ as in \eqref{eq:theta}. 

\begin{lemma}\label{lem:est1}
Let $q$ be such that \eqref{eq:est3} holds. For every $u\in\mathcal K$ it holds
\begin{equation}\label{eq:claim}
\|u\|^q_{L^q(A\setminus A_k)}\le \frac{1}{c_{N,m}(k-2)+1}\|u\|^q_{L^q(A)}
\end{equation}
with $c_{N,m}:=2^{\frac{N-m-3}{2}}$ independent of $k$.
\end{lemma}
\begin{proof}
We split the domain $A$ into $k$ subsets as follows: 
\[
A=\bigcup_{j=1}^k\Omega_{k,j},\quad\mbox{with }\Omega_{k,j}:=\left\{x\in\mathbb{R}^N:\, |x|> R, \,  \theta\in \left[\frac{j-1}{k}\cdot\frac{\pi}{2},\frac{j}{k}\cdot\frac{\pi}{2}\right] \right\}
\]
for $j=1,\dots,k$. Observe that $\Omega_{k,k}=A\setminus A_k$.
We furthermore introduce the following integrals
\[
\begin{aligned}
I_{\big[\frac{j-1}{k}\cdot\frac{\pi}{2},\frac{j}{k}\cdot\frac{\pi}{2}\big]}&:=\|u\|^q_{L^q(\Omega_{k,j})}=\omega_{N,m}\int_{R}^\infty\int_{\frac{j-1}{k}\cdot\frac{\pi}{2}}^{\frac{j}{k}\cdot\frac{\pi}{2}}|\mathfrak{u}(r,\theta)|^q\mu(\theta)r^{N-1}\,dr d\theta,
\end{aligned}
\]
where we used \eqref{eq:changeintegral}, and we note that $I_{\big[\frac{k-1}{k}\cdot\frac{\pi}{2},\frac{\pi}{2}\big]}=\|u\|^q_{L^q(A\setminus A_k)}$.
We shall now estimate from above $\|u\|^q_{L^q(A\setminus A_k)}$ with $I_{\big[\frac{j-1}{k}\cdot\frac{\pi}{2},\frac{j}{k}\cdot\frac{\pi}{2}\big]}$ for every $j=k/2+1,\dots,k-1$. To this aim, note that if $\theta\in \big[\frac{k-1}{k}\cdot\frac{\pi}{2},\frac{\pi}{2}\big]$, $\sin\theta\le \sqrt{2} \sin\left(\theta-(k-j)\frac{\pi}{2k}\right)$ for every $j=k/2+1,\dots,k-1$. Taking into account also the monotonicity of the cosine and of $\mathfrak u$ with respect to $\theta$, recalling that $\mu(\theta)=(\cos\theta)^{m-1}(\sin\theta)^{N-m-1}$, we obtain 
\[
\begin{aligned}
\|u\|^q_{L^q(A\setminus A_k)}&\le 2^{\frac{N-m-1}{2}} \omega_{N,m} \int_{R}^\infty\int_{\frac{k-1}{k}\cdot\frac{\pi}{2}}^{\frac{\pi}{2}}\left|\mathfrak u\left(r,\theta-(k-j)\frac{\pi}{2k}\right)\right|^q\\
&\hspace{.3cm}\cdot\left[\cos\left(\theta-(k-j)\frac{\pi}{2k}\right)\right]^{m-1}\left[\sin\left(\theta-(k-j)\frac{\pi}{2k}\right)\right]^{N-m-1}r^{N-1}\,dr d\theta\\
&\le 2^{\frac{N-m-1}{2}} \omega_{N,m}  \int_{R}^\infty\int^{\frac{j}{k}\cdot\frac{\pi}{2}}_{\frac{j-1}{k}\cdot\frac{\pi}{2}}\left|\mathfrak u\left(r,\theta\right)\right|^q\mu(\theta)r^{N-1}\,dr d\theta \\
&=2^{\frac{N-m-1}{2}}I_{\big[\frac{j-1}{k}\cdot\frac{\pi}{2},\frac{j}{k}\cdot\frac{\pi}{2}\big]},
\end{aligned}	
\]  
for every $j=k/2+1,\dots,k-1$. Therefore, 
\[
\begin{aligned}
\|u\|^q_{L^q(A)} & =\sum_{j=1}^k I_{\big[\frac{j-1}{k}\cdot\frac{\pi}{2},\frac{j}{k}\cdot\frac{\pi}{2}\big]}\ge \sum_{j=\frac{k}{2}+1}^{k-1} I_{\big[\frac{j-1}{k}\cdot\frac{\pi}{2},\frac{j}{k}\cdot\frac{\pi}{2}\big]} + \|u\|^q_{L^q(A\setminus A_k)}\\
& \ge \left(\frac{k-2}{2} 2^{-\frac{N-m-1}{2}}+1\right)\|u\|^q_{L^q(A\setminus A_k)},
\end{aligned}	
\]  
which provides \eqref{eq:claim}.
\end{proof}

\begin{lemma}\label{lem:willem}
Let $q\in (2,2^*)$. If $(u_n)\subset \mathcal K$ converges weakly to $u$ in $\mathcal K$, then 
\begin{equation}\label{eq:willem}
u_n|_{A_k}\to u|_{A_k}\quad \mbox{in }L^q(A_k),
\end{equation} 
with $A_k$ as in \eqref{eq:Ak}.
\end{lemma}
\begin{proof}
We shall apply \cite[Theorem 1.24]{Willem}.
As in  \cite[\S 1.5]{Willem}, given $\rho>0$, $x\in\R^N$ and $G$ a subgroup of the orthogonal group $O(N)$, we let
\begin{equation}\label{eq:m_bar_def}
\begin{multlined}
\bar{m}(x,\rho, G):=\sup\left\{ n\in\N:\, \exists\, g_1,\ldots,g_n \in G 
\right.  \text{ such that } \\
\left.  B_\rho(g_ix)\cap B_\rho(g_jx)=\emptyset \ \forall\, i\neq j \right\}.
\end{multlined}
\end{equation}
We claim that any sequence $(x_\ell)\subset\R^N$ such that
\begin{equation}\label{eq:Willem-assumption1}
\lim_{\ell\to\infty}|x_\ell|=\infty, \quad
\text{dist}(x_\ell,A_k)\leq 1 \ \forall\ell 
\end{equation}
satisfies
\begin{equation}\label{eq:Willem-assumption2}
\lim_{\ell\to\infty} \bar{m}(x_\ell,1, O(m)\times \textit{Id}_{N-m})=\infty.
\end{equation}

To this aim, we first prove that for every $(x_\ell)\subset A_k$ such that $|x_\ell|=\sqrt{s_\ell^2+t_\ell^2}\to\infty$, the sequence $(s_\ell)$ diverges as well,  where $s$ and $t$ are as in \eqref{eq:s-t}. 
Indeed, being $\theta_\ell=\arctan(t_\ell/s_\ell)$, by the definition of $A_k$ we get $t_\ell^2/s_\ell^2<\tan^2\left((1-1/k)\pi/2\right)=:\alpha^2$. Hence, 
\[
\left(1-\frac{1}{\alpha^2+1}\right)s_\ell^2=\frac{\alpha^2}{\alpha^2+1}s_\ell^2>\frac{1}{\alpha^2+1}t_\ell^2
\]
that is 
\[
s_\ell^2>\frac{1}{\alpha^2+1}(s_\ell^2+t_\ell^2)\to \infty\quad\mbox{as }\ell\to\infty,
\]
which gives the desired property. It is easy to infer that for any sequence $(x_\ell)$ satisfying \eqref{eq:Willem-assumption1} it also holds $s_\ell\to\infty$. Indeed, let $z_\ell\in A_k$ be such that $|x_\ell-z_\ell|\le 3/2$, then, as already proved, $(\sigma_\ell)^2:=(z_\ell^{(1)})^2+\dots+(z_\ell^{(m)})^2\to\infty$. Since, by the Cauchy-Schwarz inequality, $(s_\ell-\sigma_\ell)^2\le |x_\ell-z_\ell|^2$, we conclude that $s_\ell\to\infty$ as well.

We are now in a position to prove \eqref{eq:Willem-assumption2}. Let $x_\ell$ satisfy \eqref{eq:Willem-assumption1} and let $\ell$ be large.
We shall exhibit the elements $g_{j,\ell} \in G$ that appear in the definition \eqref{eq:m_bar_def} with $x=x_\ell$, $\rho=1$ and $G=O(m)\times \textit{Id}_{N-m}$.
Since $O(m)$ acts transitively on the spheres $\{\sigma \in \R^m\::\: |\sigma| = s_\ell\}$, we may choose $h_\ell \in O(m)\times \textit{Id}_{N-m}$ with the property that
  $$
  \tilde x_\ell := h_\ell x_\ell = (s_\ell,0,\dots,0,x_\ell^{N-m+1},\dots,x_\ell^{N}),
  $$
  where $x_\ell^{N-m+1},\dots,x_\ell^{N}$ denote the last $N-m$ coordinates of $x_\ell$. We now let  
\[
\phi_\ell:=2\arcsin(s_\ell^{-1}), 
\quad \text{ so that }\lim_{\ell\to\infty}\phi_\ell=0,
\]
and we let 
$$
g_{j,\ell} := \tilde g_j h_\ell \:\in \: O(m)\times \textit{Id}_{N-m},
$$
 where
$\tilde g_j$ is the $N\times N$ block matrix given by 
\begin{equation}\label{eq:g_j_def}
\tilde g_j =\begin{pmatrix}
R(j\phi_\ell) & 0 &0 \\
0  & \textrm{Id}_{m-2} & 0\\
0 &0 & \textrm{Id}_{N-m}
\end{pmatrix}.
\end{equation}
Here $R(\varphi)$ is the $2\times2$ matrix of the counterclockwise rotation of angle $\varphi$ in the plane, that is
\[
R(\varphi)=\begin{pmatrix}
\cos\varphi & -\sin\varphi \\
\sin\varphi & \cos\varphi
\end{pmatrix}.
\]
Let the indices $i,j$ be such that $0\leq i <j \leq \lfloor \pi/\phi_\ell \rfloor$.  
By \eqref{eq:g_j_def}, the definition of $\phi_\ell$ and the fact that $\phi_\ell \leq (j-i)\phi_\ell \leq \pi$, we have
\begin{multline*}
|g_{i,\ell} x_\ell - g_{j,\ell} x_\ell |^2 = |\tilde g_i \tilde x_\ell - \tilde g_j \tilde x_\ell |^2
= |\tilde x_\ell - {\tilde g_i}^{-1} \tilde g_j \tilde x_\ell |^2 \\
= 2(1-\cos((j-i)\phi_\ell)) \, s_\ell^2
=4 \sin^2\left(\frac{(j-i)\phi_\ell}{2}\right) \, s_\ell^2 \geq
4 s_\ell^2 \sin^2 \frac{\phi_\ell}{2}=4,
\end{multline*}
so that $|g_{i,\ell} x_\ell - g_{j,\ell} x_\ell |\geq 2$.  As a consequence,  $B_1(g_{i,\ell} x_\ell)\cap B_1(g_{i,\ell} x_\ell)=\emptyset$.
This proves that the matrices $g_{j,\ell}$, for $j=0, \ldots, \lfloor \pi/\phi_\ell \rfloor$ are admissible for \eqref{eq:m_bar_def}.
In conclusion,  $\lim_{\ell\to\infty} \bar{m}(x_\ell,1,O(m)\times \textit{Id}_{N-m})\geq \lim_{\ell\to\infty} \lfloor \pi/\phi_\ell \rfloor =\infty$, thus proving \eqref{eq:Willem-assumption2}.

In conclusion, by applying \cite[Theorem 1.24]{Willem}, we infer that \eqref{eq:willem} holds.
\end{proof}

As a crucial intermediate step towards Proposition \ref{prop:compact_embedding}, taking advantage of the previous two lemmas, in the next proposition we prove compactness up to the critical exponent in dimension $N$.

\begin{proposition}\label{le:subcritical_compact_embedding}
$\mathcal K$ is compactly embedded in $L^q(A)$ for every $q\in (2,2^*)$.
\end{proposition}
\begin{proof} Let $q$ be in $(2,2^*)$ and $(u_n)\subset\mathcal K$ be a sequence bounded in the $H^1(A)$-norm. By Lemma \ref{le:Kcone} there exists $u\in \mathcal K$ such that $u_n \rightharpoonup u$ in $H^1(A)$. 
Since $H^1(A)\hookrightarrow L^q(A)$ by Lemma \ref{le:continuous_embedding}, $(u_n)$ is bounded also in $L^q(A)$, and $u_n\rightharpoonup u$ in $L^q(A)$. Thus, there exists a constant $C>0$ such that for every $n$ and for every even integer $k$ 
\[
\|u_n\|_{L^q(A\setminus A_k)}^q\le \frac{1}{c_{N,m}(k-2)+1}\|u_n\|^q_{L^q(A)}\le \frac{C}{c_{N,m}(k-2)+1},
\]
where $A_k$ is as in \eqref{eq:Ak} and we used Lemma \ref{lem:est1} in the first inequality. By the weak lower semicontinuity of the norm, also the weak limit satisfies the same inequality.
Therefore, using Lemma \ref{lem:willem}, we get for every even integer $k$
\begin{equation}\label{eq:untou}
\begin{aligned}
\limsup_{n\to\infty}\|u_n-u\|_{L^q(A)}^q &=\limsup_{n\to\infty}\left(\|u_n-u\|_{L^q(A_k)}^q+\|u_n-u\|_{L^q(A\setminus A_k)}^q\right)\\
&\le \frac{C'}{c_{N,m}(k-2)+1},
\end{aligned}
\end{equation}
with $C'>0$ independent of $k$.
Finally, letting $k\to\infty$ in \eqref{eq:untou}, we have $u_n\to u$ in $L^q(A)$ and the proof is concluded. 
\end{proof}


We can now prove the compactness property $(i)$ of Theorem \ref{thm:thm2.1CM}.

\begin{proposition}\label{prop:compact_embedding}
$\mathcal K$ is compactly embedded in $L^q(A)$ for every $q\in(2,2^*_{N-m+1})$.
\end{proposition}
\begin{proof}
Let $(u_n)\subset\mathcal K$ be a sequence bounded in the $H^1(A)$-norm.  By Lemma \ref{le:Kcone} and Proposition \ref{le:subcritical_compact_embedding} there exists $u\in \mathcal K$ such that
\[
u_n \rightharpoonup u \text{ in } H^1(A) 
\quad\text{ and }\quad
u_n\to u \text{ in }L^\tau(A)
\]
for every $\tau\in (2,2^*)$. 
Let us first consider the case $m \in \{2,\ldots,N-2\}$.
Fix $\tau\in (2,2^*)$ and $q\in[2^*,2^*_{N-m+1})$. As, by Lemma \ref{le:continuous_embedding}, $u_n,u\in L^\tau(A)\cap L^{2^*_{N-m+1}}(A)$, the interpolation inequality provides
\[
\|u_n-u\|_{L^q(A)} \leq \|u_n-u\|_{L^\tau(A)}^\alpha \|u_n-u\|_{L^{2^*_{N-m+1}}(A)}^{1-\alpha},
\]
with $\alpha\in(0,1)$ such that $1/q=\alpha/\tau+(1-\alpha)/2^*_{N-m+1}$. By Proposition \ref{le:subcritical_compact_embedding}, $\|u_n-u\|_{L^\tau(A)}\to0$ as $n\to\infty$, whereas, by Lemma \ref{le:continuous_embedding}, $\|u_n-u\|_{L^{2^*_{N-m+1}}(A)}$ is bounded, thus providing the desired $L^q$-convergence.

When $m=N-1$, one can proceed in a similar way, choosing $\tau\in (2,2^*)$, $q\in[2^*,\infty)$ and applying the interpolation inequality
\[
\|u_n-u\|_{L^q(A)} \leq \|u_n-u\|_{L^\tau(A)}^\alpha \|u_n-u\|_{L^{2q}(A)}^{1-\alpha},
\]
with $1/q=\alpha/\tau+(1-\alpha)/(2q)$, as $\|u_n-u\|_{L^{2q}(A)}$ is bounded by Lemma \ref{le:continuous_embedding}.
\end{proof}

\subsection{Pointwise invariance property}
As for the property $(ii)$ in Theorem \ref{thm:thm2.1CM}, we first consider the auxiliary linear problem in the cone $\mathcal K$. 

\begin{lemma}\label{le:invarianceK}
Let $h\in \mathcal K\cap L^\infty(A)$. The linear problem 
\begin{equation}\label{Ph}
\begin{cases}
-\Delta v + v =h \quad&\mbox{in } A,\\
v \in H^1_0(A)
\end{cases}
\end{equation}
has a unique solution $v$ and moreover $v\in \mathcal K$.
\end{lemma}
\begin{proof} By interpolation, $h\in L^q(A)$ for every $q\in[2,\infty]$. Thus, by the Agmon-Douglis-Nirenberg theorem in smooth unbounded domains with bounded boundary (see for example \cite[Theorem 9.32]{brezis2010}), problem \eqref{Ph} admits a unique solution 
\[
v\in W^{2,q}(A)\cap W_0^{1,q}(A) \text{ for every } q\in [2,\infty).
\]
As a consequence, $v\in C^{1,\alpha}(\bar A)\cap L^\infty(A)$ for every $\alpha\in[0,1)$. Furthermore $v$ is a strong solution of \eqref{Ph}, meaning that $v$ and its weak derivatives satisfy the equation pointwise a.e. in $A$, see for example \cite[\S 2.4.3]{GGS}. Moreover, $\Delta v = v - h \in H^1(A)$ and so $v\in H^3(A)$.

In what follows we prove that $v\in\mathcal C$. By uniqueness and thanks to the fact that the problem (i.e., the operator, the right-hand side, and the domain) is invariant under the action of the group $O(m)\times O(N-m)$, the solution $v$ is such that $v=\mathfrak{v}(r,\theta)$. Furthermore, by the maximum principle, since $h\ge0$, also $v\ge 0$ in $A$.  

It remains to show that $\mathfrak{v}_\theta\le 0$. We observe that, being $v\in H^3(A)$, $\mathfrak{v}\in H^3_{\mathrm{loc}}((R,\infty)\times(0,\pi/2))$, so that we can perform the weak derivative in $\theta$ a.e. of the equation in \eqref{Ph}. In view of \eqref{eq:Deltau_s-r-th}, we obtain the pointwise equation
\begin{multline*}
-\left(\mathfrak{v}_{\theta rr} +\frac{N-1}{r}\mathfrak{v}_{\theta r} +\frac{1}{r^2} \mathfrak{v}_{\theta\theta\theta}-\frac{(m-1)\tan\theta - (N-m-1)\cot\theta}{r^2}\mathfrak{v}_{\theta\theta} \right)+ \\
+ \frac{\mathfrak{v}_\theta}{r^2}\left(\frac{m-1}{\cos^2\theta}+\frac{N-m-1}{\sin^2\theta}\right) + \mathfrak{v}_\theta=\mathfrak{h}_\theta
\end{multline*}
for $(r,\theta)\in (R,\infty)\times\left(0,\pi/2\right)$.
Going back to the $x$-variable, we denote
\[
v_\theta(x):=\mathfrak{v}_\theta(r,\theta) \quad\mbox{ and }\quad
h_\theta(x):=\mathfrak{h}_\theta(r,\theta),
\]
so that we can rewrite the previous equation as
\[
-\Delta v_\theta + \left(\frac{m-1}{x_1^2+\ldots+x_m^2}+\frac{N-m-1}{x_{m+1}^2+\ldots+x_n^2}+1\right) v_\theta = h_\theta\\
\]
in
\[ 
\tilde A:= \{ x \in A\::\: x_1^2+\ldots+x_{m}^2 \not = 0 \mbox{ and }x_{m+1}^2+\ldots+x_{N}^2 \not = 0\}.
\]
We claim that the H\"older continuous function $v_\theta$ vanishes on $\partial \tilde A$. Indeed, first note that
\[
\begin{aligned}
\partial \tilde A&=\partial A\cup\{x\in A\,:\,x_1=\dots=x_m=0\}\cup\{x\in A\,:\,x_{m+1}=\dots=x_N=0\}\\
&=\partial A\cup(A\cap\{s=0\})\cup(A\cap\{t=0\}),
\end{aligned}
\] 
with $s,t$ as in \eqref{eq:s-t}.
Since $v=\mathfrak v(r,\theta)$, $v_\theta$ coincides with the tangential derivative of $v$ on $\partial A$; hence, the boundary condition $v|_{\partial A}=0$ implies $v_\theta=0$ on $\partial A$.
The fact that $v_\theta=0$ on the remaining part of $\partial \tilde A$ follows noticing that for every $i=1,\dots,N$, 
\[v_{x_i}(x_1,\dots,x_{i-1},0,x_{i+1},\dots,x_N)=0\] 
by the symmetry and regularity of $v$, and that $\mathfrak{v}_\theta=-t\tilde{\mathfrak{v}}_s+s\tilde{\mathfrak{v}}_t$, cf. Subsection \ref{subsec:chofvar}.
   
In sum, $v_\theta$ satisfies the problem
\begin{equation}\label{pb:for-u_theta}
 \begin{cases}
-\Delta v_\theta+\left(\frac{m-1}{x_1^2+\ldots+x_{m}^2}+\frac{N-m-1}{x_{m+1}^2+\ldots+x_{N}^2}+1\right)v_\theta = h_\theta\le 0 & \mbox{ a.e. in } \tilde A\\
v_\theta \in H^1_0(\tilde A) &  
\end{cases}
\end{equation}
Due to the singularity of the equation on the set $\{s=0\}\cup\{t=0\}$, we cannot apply the weak maximum principle directly to deduce that $v_\theta \le 0$. Instead, we let $\varepsilon>0$, and consider the function $w:= (v_\theta-\varepsilon)^+\in  H^1_0(\tilde A)$ and a radial, radially nonincreasing cutoff function $\eta \in C^\infty_c(B_{2}(0))$ with the following properties
\[
0\le \eta \le 1,\quad \eta\equiv 1 \mbox{ in } B_1(0).
\]
Let $w_n(x):=\eta(x/n) w(x)$ for every $x\in \tilde A$. Notice that $(w_n)$ approximates $w$ in the $H^1$-norm. Hence, 
$\int_{\tilde A}\nabla v_\theta\cdot\nabla w_n\,dx\to \int_{\tilde A}\nabla v_\theta\cdot\nabla w\,dx$ and
\begin{equation}\label{eq:monotone}
\begin{aligned}
&\lim_{n\to\infty}\int_{\tilde A} \Bigl(\frac{m-1}{x_1^2+\ldots+x_{m}^2}+\frac{N-m-1}{x_{m+1}^2+\ldots+x_{N}^2}+1\Bigr)v_\theta w_n\,dx \\
&=\lim_{n\to\infty} \int_{\tilde A\cap\{v_\theta>\varepsilon\}} \Bigl(\frac{m-1}{x_1^2+\ldots+x_{m}^2}+\frac{N-m-1}{x_{m+1}^2+\ldots+x_{N}^2}+1\Bigr) \cdot\\
& \hspace{7truecm} \cdot v_\theta(v_\theta - \varepsilon)\eta\left(\frac{x}{n}\right)\,dx\\
&=  \int_{\tilde A} \Bigl(\frac{m-1}{x_1^2+\ldots+x_{m}^2}+\frac{N-m-1}{x_{m+1}^2+\ldots+x_{N}^2}+1\Bigr)v_\theta w\,dx,
\end{aligned}
\end{equation}
where the limit under the integral sign is justified by monotone convergence, since the integrand is nonnegative in $\{v_\theta>\varepsilon\}$. 

Now we multiply \eqref{pb:for-u_theta} by the function $w_n$ whose support is compactly contained in the set $\tilde A$ where the equation is non-singular. Hence, we can integrate by parts, obtaining the inequality
\begin{equation*}
\begin{aligned}
0&\ge \int_{\tilde A} h_\theta w_n \,dx = \int_{\tilde A}\Bigl( -\Delta v_\theta+\Bigl(\frac{m-1}{x_1^2+\ldots+x_{m}^2}+\frac{N-m-1}{x_{m+1}^2+\ldots+x_{N}^2}+1\Bigr)v_\theta\Bigr) w_n\,dx \\
&= \int_{\tilde A}\Bigl(\nabla v_\theta \cdot \nabla w_n   +\Bigl(\frac{m-1}{x_1^2+\ldots+x_{m}^2}+\frac{N-m-1}{x_{m+1}^2+\ldots+x_{N}^2}+1\Bigr)v_\theta w_n\Bigr)\,dx.
\end{aligned}
\end{equation*}
By combining the last inequality with \eqref{eq:monotone}, we can pass to the limit and obtain 
\begin{align*}
0&\ge \int_{\tilde A}\Bigl(\nabla v_\theta \cdot \nabla w   +\Bigl(\frac{m-1}{x_1^2+\ldots+x_{m}^2}+\frac{N-m-1}{x_{m+1}^2+\ldots+x_{N}^2}+1\Bigr)v_\theta w\Bigr)\,dx \nonumber \\
&\ge \int_{\tilde A}\Bigl(|\nabla w|^2   +\Bigl(\frac{m-1}{x_1^2+\ldots+x_{m}^2}+\frac{N-m-1}{x_{m+1}^2+\ldots+x_{N}^2}+1\Bigr)w^2\Bigr)\,dx
\end{align*}
This implies $w = (v_\theta-\varepsilon)^+ \equiv 0$ in $\tilde A$. Since $\varepsilon>0$ was chosen arbitrarily, we deduce that $v_\theta\le 0$, thus concluding the proof.
\end{proof}

By combining Lemma \ref{le:invarianceK} with a truncation argument, we can conclude the proof of the pointwise invariance of $\mathcal K$.

\begin{proposition}\label{prop:invarianceC}
Let $u\in \mathcal K$. The problem 
\begin{equation}\label{P-up-1}
\begin{cases}
-\Delta v + v =a(x)u^{p-1}\quad&\mbox{in } A,\\
v \in H^1_0(A)
\end{cases}
\end{equation}
has a unique solution $v$ and moreover $v\in \mathcal K$.
\end{proposition}
\begin{proof}
We first prove the result with the additional hypothesis $u\in L^\infty(A)$. Since both $a$ and $u$ belong to $\mathcal C\cap L^\infty(A)$, $au^{p-1}\in\mathcal C\cap L^\infty(A)$. Moreover, being  by interpolation $u\in L^q(A)$ for every $q\in[2,\infty]$, $au^{p-1}\in L^2(A)$. It remains to show that $|\nabla(au^{p-1})|\in L^2(A)$. For this, we use the assumptions \eqref{a_assumptions} on $a$ and that $u\in L^\infty(A)$ to get 
\[
\begin{split}
\int_A |\nabla(au^{p-1})|^2\,dx \le 2\int_A |\nabla a|^2 u^{2(p-1)}\,dx+2(p-1)^2\int_A a^2 u^{2(p-2)}|\nabla u|^2\,dx\\
\le 2\|\nabla a\|_{L^2(A)}^2\|u\|_{L^\infty(A)}^{2(p-1)}+2(p-1)^2 \|a\|_{L^\infty(A)}^2\|u\|_{L^\infty(A)}^{2(p-2)}\|\nabla u\|_{L^2(A)}^2 < \infty.
\end{split}
\]
Then, the conclusion follows by Lemma \ref{le:invarianceK}, with $h=au^{p-1}\in \mathcal K\cap L^\infty(A)$.

The remaining part of the proof takes inspiration from \cite[Proposition 4.2]{CowanMoameni2022}. Let $u\in \mathcal K$ and define $u_n:=\min\{u,n\}$ for every $n\in\mathbb N$. Then, $(u_n)\subset \mathcal K\cap L^\infty(A)$. By the first part of the proof, for every $u_n$ there exists a unique $v_n\in\mathcal K$ solving 
\[
\begin{cases}
-\Delta v + v =a(x)u_n^{p-1}\quad&\mbox{in } A,\\
v \in H^1_0(A).
\end{cases}
\]
Hence, using the equation solved by $v_n$, H\"older's inequality, and that $\mathcal K\hookrightarrow L^p(A)$ by Lemma \ref{le:continuous_embedding}, 
\begin{equation}\label{eq:est-vn}
\begin{aligned}
\|v_n\|_{H^1(A)}^2&= \int_A a(x)u_n^{p-1}v_n\,dx\le \|a\|_{L^\infty(A)}\|u_n\|_{L^p(A)}^{p-1}\|v_n\|_{L^p(A)}\\
&\le C\|a\|_{L^\infty(A)}\|u_n\|_{L^p(A)}^{p-1}\|v_n\|_{H^1(A)}.
\end{aligned}
\end{equation}
Now, being $u_n\to u$ a.e. in $A$ and $0\le u_n\le u_{n+1}$ for every $n$, by monotone convergence, 
\[
\lim_{n\to\infty}\int_A u_n^p\,dx=\int_A u^p\,dx.
\]
Therefore, $(u_n)$ is bounded in $L^p(A)$, and by \eqref{eq:est-vn},
\[
\|v_n\|_{H^1(A)}\le C\|a\|_{L^\infty(A)}\|u_n\|_{L^p(A)}^{p-1}\le C'\|a\|_{L^\infty(A)}, 
\]
with $C'>0$ independent of $n$. Thus, $(v_n)$ is bounded in $H^1_0(A)$, and so, there exists $v\in H^1_0(A)$ such that up to a subsequence $v_n\rightharpoonup v$ weakly in $H^1(A)$. This implies that for every $\varphi\in C^\infty_c(A)$, 
\[
\lim_{n\to\infty}\int_A(\nabla v_n\cdot\nabla \varphi+v_n\varphi)\,dx = \int_A(\nabla v\cdot\nabla \varphi+v\varphi)\,dx. 
\]
Moreover, for every $\varphi\in C^\infty_c(A)$ it also holds, by monotone convergence, 
\[
\lim_{n\to\infty}\left|\int_A a(x)\left(u_n^{p-1}-u^{p-1}\right) \varphi\,dx \right|\le \lim_{n\to\infty}\|a\varphi\|_{L^\infty(A)} \int_A \left(u^{p-1} - u_n^{p-1}\right)\,dx = 0. 
\]
Whence, being for every $n\in\mathbb N$, 
\[
\int_A(\nabla v_n\cdot\nabla \varphi+v_n\varphi)\,dx = \int_A a(x)u_n^{p-1}\varphi\,dx\;\mbox{ for every }\varphi\in C^\infty_c(A),
\]
we get 
\[
\int_A(\nabla v\cdot\nabla \varphi+v\varphi)\,dx = \int_A a(x)u^{p-1}\varphi\,dx\;\mbox{ for every }\varphi\in C^\infty_c(A)
\]
that is $v\in H^1_0(A)$ solves \eqref{P-up-1}. Finally, since $\mathcal K$ is weakly closed by Lemma \ref{le:Kcone}, $v\in \mathcal K$ and the proof is concluded.
\end{proof}

\subsection{End of the proof of Theorem \ref{thm:main_existence}}
\label{sec:end-proof-theorem}

As already mentioned, for the existence part we apply Theorem \ref{thm:thm2.1CM} with $\Omega=A$, $w=a$, and $K=\mathcal{K}$. We observe that all the hypotheses therein are satisfied. Indeed, $\mathcal K$ is a convex set which is closed in $H^1_0(A)$ by Lemma \ref{le:Kcone}. As for the compactness property (i), we know that from every $H^1$-bounded sequence $(u_n)\subset\mathcal K$ we can extract a subsequence weakly convergent in $H^1(A)$; since
$p \in (2,2^*_{N-m+1})$, by 
Proposition \ref{prop:compact_embedding} $(u_n)$ strongly converges in $L^p(A)$.

Moreover, the pointwise invariance property (ii) is satisfied by Proposition \ref{prop:invarianceC}. Finally, any radial and strictly positive function $v \in H^1_0(A)$ belongs to ${\mathcal K}$ and satisfies $\int_A a|v|^pdx>0$ by assumption (\ref{a_assumptions}), which implies that $I(\lambda v)\le 0$ for sufficiently large $\lambda >0$. Hence assumption (iii) of Theorem \ref{thm:thm2.1CM} is also satisfied. 
Consequently, Theorem \ref{thm:thm2.1CM} guarantees the existence of a nontrivial weak solution $u$ of \eqref{P} belonging to $\mathcal K$ and having the following variational characterization 
\[
I(u)=\inf_{\gamma\in\Gamma}\sup_{t\in[0,1]} I(\gamma(u))>0,
\]
where $\Gamma:=\{\gamma\in C([0,1];\mathcal K):\gamma(0)=0\neq\gamma(1),\,I(\gamma(1))\le 0\}$. 
We observe that, being $u$ a solution in $\mathcal K\setminus\{0\}$, $I'(u)u=0$, and so $u\in\mathcal N_\mathcal K$.
 
Now, arguing as in \cite[Section 2.4]{BCNW2023}, $c_I:=\inf_{\mathcal N_{\mathcal K}}I(u)=\inf_{u\in\mathcal K\setminus \{0\}}\sup_{t\ge 0}I(tu)$. Moreover, by virtue of the fact that $a(x)t^{p-1}/t$ is strictly increasing in $t$, the equality $\inf_{u\in\mathcal K\setminus \{0\}}\sup_{t\ge 0}I(tu)=\inf_{\gamma\in\Gamma}\sup_{t\in[0,1]} I(\gamma(u))$ holds as well, cf. for instance a proof for a similar case contained in \cite[Lemma 5.4]{CN}. Altogether, the solution found has energy level $c_I$.  \hfill $\qed$

\section{The case of radial weight $a$}
\label{sec:case-a-radial}

The aim of this section is to prove the multiplicity result stated in Theorem \ref{thm:main-a-rad}.  We deal with problem \eqref{P-rad}, that is
\begin{equation*}
   \begin{cases}
  -\Delta u + u = a(|x|) u^{p-1} \qquad &\mbox{in }A,\\
u>0 &\mbox{in }A,\\
u \in H^1_0(A) &
  \end{cases}
\end{equation*} 
with $a$ being a radial weight satisfying \eqref{a_assumptions} and $A = \{x\in\mathbb{R}^N:\, |x|>{R}\}$.

As a first step, we fix any integer $m \in \{2,\ldots,N-1\}$, and we show in Proposition \ref{prop:non_radial} below that, for $p$ satisfying \eqref{eq:suff_p},  every $\mathcal K$-ground state solution of \eqref{P-rad} is nonradial.

To this end, let $\eta(x)=\mathfrak y(\theta)$ be given by
\begin{equation}\label{eq:eta-def}
\mathfrak y(\theta):=(N-m)\cos^2\theta-m\sin^2\theta 
=\frac{N}{2}(\cos(2\theta)+1)-m,
\end{equation}
for $\theta\in(0,\pi/2)$,  or equivalently, for $x\neq0$,
\[
\eta(x)=(N-m)\sum_{i=1}^m \left(\frac{x_i}{|x|}\right)^2 - m\sum_{i=m+1}^N \left(\frac{x_i}{|x|}\right)^2.
\]
In view of the above expression, it is possible to show that 
\begin{equation}\label{radial-weakly-stable-assumption-4}
    \int_{\mathbb S_r^{N-1}} \eta(x)\,d\sigma = 0\quad\mbox{for every $r>0$,}
\end{equation}
where $\mathbb S_r^{N-1}$ is the sphere of radius $r$ of $\mathbb R^N$. 
Indeed, as for every $i,j=1,\ldots,N$, and $r>0$
\[
\int_{\mathbb S_r^{N-1}} x_i^2\,d\sigma = \int_{\mathbb S_r^{N-1}} x_j^2\,d\sigma,
\]
it results that
\[
\int_{\mathbb S_r^{N-1}} \eta(x)\,d\sigma 
= \frac{1}{r^2}\int_{\mathbb S_r^{N-1}}\left((N-m)\sum_{i=1}^m x_i^2-m\sum_{i={m+1}}^N x_i^2 \right)d\sigma 
 = 0.
\]
By explicit calculations it is possible to verify that $\mathfrak y$ satisfies the following condition
\begin{equation}\label{radial-weakly-stable-assumption-2}
\mathfrak{y}'(\theta) \le 0 \qquad\qquad \text{ for $\theta\in (0,\pi/2)$,}
\end{equation} and that it solves the boundary value problem
\begin{equation}\label{eq:eta}
\begin{cases}
-(\mu(\theta)\mathfrak y')'=2N \mu(\theta) \mathfrak y \qquad & \mbox{in } \left(0,\frac{\pi}{2}\right) \\
\mathfrak y'(0)=\mathfrak y'\left(\frac{\pi}{2}\right)=0,
\end{cases}
\end{equation}
where $\mu$ is given in \eqref{eq:mu}.

The next lemma takes inspiration from both \cite{BCNW2023} and \cite{CowanMoameni2022}.  

\begin{lemma}\label{lem:v_def}
Assume that $p$ satisfies \eqref{eq:suff_p}.
Let $u_\mathrm{rad}(x)=\mathfrak u_\mathrm{rad}(r)\in\mathcal K$ be a nontrivial radial solution of \eqref{P-rad}.
Then the function $v(x)=\mathfrak{v}(r,\theta):=\mathfrak u_\mathrm{rad}(r)\mathfrak y(\theta)$, with $\mathfrak y$ as in \eqref{eq:eta-def}, belongs to $H^1_0(A)\cap L^p(A)$ and satisfies the following property
\begin{equation}
\label{radial-weakly-stable-assumption-1}
I''(u_\mathrm{rad})(v,v)<0. 
\end{equation}
\end{lemma}
\begin{proof}
As $u_\mathrm{rad} \in \mathcal K \subset L^p(A)$ by Lemma \ref{le:continuous_embedding} and $\mathfrak{y}$ is bounded, $v\in L^p(A)$.

In order to prove \eqref{radial-weakly-stable-assumption-1},  we take advantage of \eqref{eq:gradient-r-theta} to write
\[
\begin{aligned}
&I''(u_\mathrm{rad})(v,v)=\int_A \left( |\nabla v|^2+v^2-(p-1)a(|x|) u_\mathrm{rad}^{p-2}v^2\right)\,dx\\ 
&= \omega_{N,m}\int_{R}^\infty\bigg\{\int_0^{\frac{\pi}{2}}\bigg[\mathfrak{u}_\mathrm{rad}'^{\,2}(r)\mathfrak y^2(\theta)+\frac{\mathfrak{u}_\mathrm{rad}^2(r)\mathfrak y'^{\,2}(\theta)}{r^2}\bigg.\bigg.\\
&\bigg.\bigg.\hspace{2cm}+\mathfrak{u}_\mathrm{rad}^2(r)\mathfrak y^2(\theta)\Big(1-(p-1)\mathfrak a(r)\mathfrak{u}_\mathrm{rad}^{p-2}(r)\Big)\bigg]\mu(\theta) r^{N-1}\,d\theta\bigg\}dr\\
&  = \omega_{N,m}\left(\int_{R}^\infty\left[ \mathfrak{u}_\mathrm{rad}'^{\,2}(r)+\mathfrak{u}_\mathrm{rad}^{\,2}(r)-(p-1)\mathfrak a(r) \mathfrak{u}_\mathrm{rad}^{p}(r)\right]r^{N-1}\, dr\right) \cdot \left(\int_0^{\frac{\pi}{2}} \mathfrak y^2(\theta)\mu(\theta)\,d\theta\right) \, \\
& \hspace{2cm}+  \omega_{N,m}\left(\int_{R}^\infty \frac{\mathfrak{u}_\mathrm{rad}^2(r)}{r^2} r^{N-1}\, dr\right)
\cdot \left(\int_0^{\frac{\pi}{2}} \mathfrak y'^{\,2}(\theta)\mu(\theta)\,d\theta\right),
\end{aligned}
\]
with $\omega_{N,m}$ as in \eqref{eq:changeintegral}.
Then we notice that \eqref{eq:eta} implies
\begin{equation}\label{eq:eta-integrated}
\int_0^{\frac{\pi}{2}} \mathfrak y'^{\,2}(\theta)\mu(\theta) \, d\theta = 
2N \int_0^{\frac{\pi}{2}} \mathfrak y^2(\theta) \mu(\theta) \, d\theta.
\end{equation}
Moreover, since $u_\mathrm{rad}$ solves \eqref{P-rad}, it satisfies
\begin{equation}\label{eq:u-rad-integrated}
\int_{R}^{\infty} \left[ \mathfrak{u}_\mathrm{rad}'^{\,2}(r) +\mathfrak{u}_\mathrm{rad}^{2}(r) \right] r^{N-1}\, dr= \int_{R}^{\infty} \mathfrak a(r) \mathfrak{u}_\mathrm{rad}^p(r) r^{N-1}\, dr.
\end{equation}
Altogether, we find
\[
\begin{aligned}
I''(u_\mathrm{rad})(v,v) &=
\omega_{N,m} \left(\int_{R}^{\infty} \bigg\{ -(p-2)\bigg[ \mathfrak{u}_\mathrm{rad}'^{\,2}(r) +\mathfrak{u}_\mathrm{rad}^{2}(r) \bigg]+2N \frac{\mathfrak{u}_\mathrm{rad}^2(r)}{r^2}  \bigg\} r^{N-1} \, dr\right) \\
&\hspace{7.5cm} \cdot \left(\int_0^{\frac{\pi}{2}} \mathfrak y^2(\theta) \mu(\theta) \, d\theta\right).
\end{aligned}
\]
Finally, Hardy's inequality 
\[
\int_{R}^\infty \frac{\mathfrak{u}_\mathrm{rad}^2(r)}{r^2}r^{N-1}dr\le \left(\frac{2}{N-2}\right)^2\int_{R}^\infty \mathfrak{u}_\mathrm{rad}'^{\,2}(r)r^{N-1}dr,
\]
allows to estimate
\[
\begin{aligned}
I''(u_\mathrm{rad})(v,v)& <
\omega_{N,m}\left\{2N-(p-2)\left[\left(\frac{N-2}{2}\right)^2+R^2\right]\right\}\\
&\hspace{3cm}\cdot \left(\int_{R}^\infty \frac{\mathfrak{u}_\mathrm{rad}^2(r)}{r^2} r^{N-1}dr\right)\cdot\left(\int_0^{\frac{\pi}{2}} \mathfrak y^2(\theta) \mu(\theta) \,d\theta\right),
\end{aligned}
\]
where the inequality is strict since $R>0$ and $u_\mathrm{rad}\neq 0$.
In view of \eqref{eq:suff_p}, the above inequality provides \eqref{radial-weakly-stable-assumption-1}.
The previous calculations also show that $v\in H^1_0(A)$, thus concluding the proof.
\end{proof}

As already mentioned, problem \eqref{P-rad} admits a radial solution for every $p>2$. We show in the next proposition that any solution provided by Theorem \ref{thm:main_existence} for $m \in \{2,\ldots,N-1\}$ is nonradial.

\begin{proposition}\label{prop:non_radial}
Let $N\geq3$ and let $p$ satisfy \eqref{eq:suff_p}.
Then every $\mathcal K$-ground state solution of \eqref{P-rad} is nonradial.
\end{proposition}
\begin{proof}
Let $u$ be a $\mathcal K$-ground state solution of \eqref{P-rad}. Suppose by contradiction that $u$ is radial so that Lemma \ref{lem:v_def} applies to the function $v=u\eta=\mathfrak u(r)\mathfrak{y}(\theta)$, with $\eta=\mathfrak{y}(\theta)$ as in \eqref{eq:eta-def}.
By (\ref{radial-weakly-stable-assumption-1}) and the continuity of $I''$, there exist $\delta \in (0,1)$ and $\rho>0$ with the property that 
\begin{equation}
  \label{second-der-stable-pointw-est}
\qquad  \qquad I''\bigl(t(u+\tau v)\bigr)(v,v)<0 \qquad \text{for $t \in
[1-\delta,1+\delta]$, $\tau \in [-\rho,\rho]$.}
\end{equation}
By adjusting $\delta$ and $\rho$ if necessary, we have that
$$
t(u+\tau v)=tu(1+\tau \eta) \ge 0  \quad \text{in $A$} \qquad \text{for $t \in
[1-\delta,1+\delta]$, $\tau \in [-\rho,\rho]$.}
$$
Combining this information with (\ref{radial-weakly-stable-assumption-2}) we deduce that 
$$
t(u+\tau v) \in \mathcal K \qquad \text{for $t \in
[1-\delta,1+\delta]$, $\tau \in [0,\rho]$.}
$$
Moreover, since $I'\bigl((1-\delta)u\bigr)u>0>I'\bigl((1+\delta)u\bigr)u$, cf. for instance \cite[Proposition~4.1]{BCNW2023}, there exists $s_* \in (0,\rho)$ for which  
$$
I'\bigl((1-\delta)(u+s_*v)\bigr)(u+s_*v)>0>I'\bigl((1+\delta)(u+s_*v)\bigr)(u+s_*v).
$$
By the intermediate value theorem, there exists 
$t_* \in [1-\delta,1+\delta]$ with 
$$
I'\bigl(t_*(u+s_*v)\bigr)(u+s_*v)=0,
$$
and therefore
\begin{equation}\label{eq:u*N}
u_*:= t_*(u+s_*v) \in {\mathcal N}_{\mathcal{K}}.
\end{equation}
Since the inequality $I(u)\ge I(tu)$ holds for every $t\ge 0$
(cf. \cite[Proposition 4.1]{BCNW2023}), by a first order Taylor expansion with integral reminder and \eqref{second-der-stable-pointw-est}, we have 
\begin{align*}
I(u_*) -I(u)&\le I(u_*) -I(t_*u)\\
&= s_* t_* I'(t_*\, u)v +  t_*^2 \int_0^{s_*}
                                               I''(t_*(u+\tau v))(v,v)(s_*-\tau) \:d\tau\\
&< s_* t_* I'(t_*\, u)v.
\end{align*}
Now, using \eqref{eq:I_def}, the equation satisfied by $u$, and \eqref{radial-weakly-stable-assumption-4},
\begin{align*}
I(u_*) -I(u)& < s_*t_*\left( t_* \langle u, v \rangle_{H^1(A)} - t_*^{p-1} \int_{A} a(|x|)u^{p-1} v\,dx \right)\\
&= s_*t_*^2(1-t_*^{p-2})\int_{R}^{\infty} \mathfrak{a}(r) \mathfrak{u}^{p}(r) \left( \int_{\mathbb S_r^{N-1}} \eta(x) \,d\sigma \right) \,dr = 0.
\end{align*}
Consequently,  by \eqref{eq:u*N}, $c_I \le I(u_*)<I(u)=c_I$ gives the desired contradiction.
\end{proof}

In order to obtain multiplicity, it remains to show that different values of $m$ provide different solutions.

\begin{lemma}\label{lem:m_m'}
Let $m,m'$ be two integers such that $2\leq m<m'\leq N-1$. 
Let $u$ be a function defined a.e.  in $A$ having both symmetries
\begin{align}
u(x) & = \tilde{\mathfrak{u}}_m \left( \sqrt{x_1^2+\ldots+x_m^2},  \sqrt{x_{m+1}^2+\ldots+x_N^2}  \right) 
\label{eq:mm'1} \\
& = \tilde{\mathfrak{u}}_{m'} \left( \sqrt{x_1^2+\ldots+x_{m'}^2},  \sqrt{x_{m'+1}^2+\ldots+x_N^2}  \right),
\label{eq:mm'2}
\end{align}
then $u$ is radial in $A$.
\end{lemma}
\begin{proof} For the reader's convenience, we report here the proof that can also be found in  \cite[Lemma 4.1]{CowanMoameni2022JDE}.
Let 
\[
\Pi:=\{x\in A: \, x_2=\ldots=x_{m'-1}=x_{m'+1}=\ldots=x_N=0\}.
\]
We infer from \eqref{eq:mm'2} that, for every $x\in \Pi$, it holds
\[
u(x)=\tilde{\mathfrak{u}}_{m'}\left(\sqrt{x_1^2+x_{m'}^2},0\right).
\]
As, for $x\in\Pi$, $|x|=\sqrt{x_1^2+x_{m'}^2}$, this implies that the restriction of $u$ to $\Pi$ is radial, that is
\begin{equation}\label{eq:u-radial}
u(x)= u(|x|) \qquad \text{for } x\in\Pi.
\end{equation}
On the other hand, for every $x=(x_1,\ldots,x_N)\in A$, we let $\tilde{x}=(\tilde{x}_1,\ldots,\tilde{x}_N)$ be given by
\[
\tilde{x}_1=\sqrt{x_1^2+\ldots+x_m^2},
\quad
\tilde{x}_{m'}=\sqrt{x_{m+1}^2+\ldots+x_N^2},
\quad
\tilde{x}_\ell=0 \ \ \forall \, \ell\neq 1,m'.
\]
We notice that $\tilde{x}\in \Pi$ and that $|x|=|\tilde{x}|$. 
Being $m<m'$, relation \eqref{eq:mm'1} provides 
\[
u(x)=\tilde{\mathfrak u}_m(\tilde{x}_1,\tilde{x}_{m'})=u(\tilde{x}).
\]
Thus, combining with \eqref{eq:u-radial}, we obtain
\[
u(x)=u(\tilde{x})= u(|\tilde{x}|)= u(|x|)
\]
for every $x\in A$, namely $u$ is radial.
\end{proof}

\begin{proof}[$\bullet$ Proof of Theorem \ref{thm:main-a-rad}] 
Let $2\leq n \leq N-1$ be an integer.
Being $a$ radial,  problem \eqref{P-rad} admits a radial solution for every $p>2$,  see the Introduction.
It remains to exhibit $n-1$ nonradial rotationally nonequivalent solutions.

Since $a$ is radial, $a\in\mathcal C$ for every value of $m\in\{2,\ldots,N-1\}$. Then Theorem~\ref{thm:main_existence} applies for any value of $m$ that ensures 
\begin{equation}\label{eq:treshold}
p < 2^*_{N-m+1}.
\end{equation}
Since the map $n\mapsto 2^*_n$ is strictly decreasing,  the condition $p<2^*_n$ implies \eqref{eq:treshold} provided that $m\in\{N-n+1,\ldots,N-1\}$.
Therefore, Theorem \ref{thm:main_existence} applies with $n-1$ choices of $m$. 
By Proposition \ref{prop:non_radial} and Lemma \ref{lem:m_m'} the solutions found are nonradial and rotationally nonequivalent.
\end{proof}

\begin{proof}[$\bullet$ Proof of Corollary \ref{coro:multiplicity}] 
\noindent $\mathrm{(i)}$ Let $n=2$.
Since $2^*_2=\infty$ and $R^*=0$ for $p$ satisfying \eqref{eq:coro1},  Theorem \ref{thm:main-a-rad} ensures that, for every inner radius $R>0$, problem \eqref{P-rad} admits two solutions, one of which is radial and the other is not.
\par
\smallskip
\noindent $\mathrm{(ii)}$ Let $n\geq3$ integer.
The inequality 
\[
2+\frac{8N}{(N-2)^2} < 2^*_n
\]
is equivalent to $n < N/2 + 2/N$, that is, being $n$ integer, $n \leq N/2$.  
Notice that the condition $N \geq 6$ appearing in the statement ensures that the choice $n = 3$ is admissible (the case $n=2$ is treated in (i)). 
Summing up, for $N \geq 6$ and $3\leq n\leq N/2$,  the interval 
\[
\left(2+\frac{8N}{(N-2)^2} , 2^*_n \right)
\]
is non-empty.  
For $p$ in this interval, $R^*=0$, so that Theorem \ref{thm:main-a-rad} applies providing $n$ rotationally nonequivalent solutions of \eqref{P-rad} for every value of  $R>0$.
\end{proof}

\section{Existence in nonradial exterior domains}
\label{sec:exist-nonr-exter}

In this section, we sketch the proof of Theorem~\ref{thm:main_existence-nonradial}. Henceforth, we fix $m \in \{2,\dots,N-1\}$ and a $C^2$-function $g:[0,\infty) \to \R$ satisfying \eqref{eq:assumption-g}, and we let the exterior domain $A_g$ be given by \eqref{eq:def-A_g}. It easily follows from \eqref{eq:assumption-g}  that $A_g$ is a domain of class $C^2$ with boundary given by 
$$
\partial A_g= \{x \in \R^N\::\: x_1^2+ \dots + x_m^2 + g(x_{m+1}^2 + \cdots + x_N^2) =0\},
$$
and that $\R^N \setminus A_g$ is a compact subset containing $B_R(0)$ for sufficiently small $R>0$. In the following, we fix $R>0$ with this property, and we let $A := \{x \in \R^N\::|x| > R\}$. Then we have $A_g \subset A$ and therefore $H^1_0(A_g) \subset H^1_0(A)$ by trivial extension, which implies that $\mathcal C_g \subset \mathcal C$ and $\mathcal K_g \subset \mathcal K$.

As a consequence, Lemma~\ref{le:continuous_embedding} and Proposition~\ref{prop:compact_embedding} directly extend to $A_g$ and $\mathcal K_g$ in place of $A$ and $\mathcal K$. 

To extend Proposition~\ref{prop:invarianceC}, we need the following important property which is guaranteed by the upper bound on $g'$ in assumption~(\ref{eq:assumption-g}).

\begin{lemma}
  \label{positivity-boundary}
If $v \in C^1(\overline{A_g})$ satisfies $v \equiv 0$ on $\partial A_g$ and
$v \ge 0$ in $A_g$, then $v_\theta \le 0$ on $\partial A_g$, where, as before, $v_\theta$ denotes the derivative of $v$ with respect to the angle $\theta$.
\end{lemma}

\begin{proof} In the following, for $x \in \R^N$ we write
  $x = x^s + x^t$ with
  $$
  x^s= (x_1,\dots,x_m,0,\dots,0) \quad \text{and}\quad  x^t = (0,\dots,0,x_{m+1},\dots,x_N) \in  \R^N,
  $$
  so that
  $|x^s|=s$ and $|x^t|=t$ with $s$ and $t$ given in (\ref{eq:s-t}). By \eqref{eq:assumption-g} and \eqref{eq:def-A_g},
a normal vector field
on $\partial A_g$ pointing inside $A_g$ is given by
$$
\mu: \partial A_g \to \R^N, \qquad \mu(x)= x^s + g'(t)x^t.
$$
By assumption, we have $\nabla v(x) = c(x) \mu(x)$ for $x \in \partial A_g$ with some nonnegative function $c$, and therefore, by 
\eqref{eq:assumption-g},
\begin{equation*}
v_\theta(x)= \nabla v(x) \cdot \Bigl(-\frac{t}{s}x^s + \frac{s}{t}x^t\Bigr)= c(x)\mu(x)\cdot \Bigl(-\frac{t}{s}x^s + \frac{s}{t}x^t\Bigr)= c(x)st \bigl(-1+g'(t)\bigr) \le 0
\end{equation*}
for $x \in \partial A_g$, as claimed.
\end{proof}

With the help of this lemma, we can now prove Lemma~\ref{le:invarianceK} and Proposition~\ref{prop:invarianceC} in exactly the same way as before for $A_g$ and $\mathcal K_g$ in place of $A$ and $\mathcal K$, noting that we only used the fact that $v_\theta \le 0$ on $\partial A$ for the solutions $v$ appearing in Lemma~\ref{le:invarianceK} and Proposition~\ref{prop:invarianceC} and not the equality. As a consequence, we may verify the assumptions (i) and (ii) of Theorem~\ref{thm:thm2.1CM} precisely as in Section~\ref{sec:end-proof-theorem}. To verify (iii), we argue as follows. Let $u \in \mathcal K_g \setminus \{0\}$ be arbitrary, and let $v$ be the unique solution of~(\ref{P-up-1}) with $A$ replaced by $A_g$. Then $v \in \mathcal K_g$, and $v$ is strictly positive in $A_g$ by the strong maximum principle. Hence $\int_{A_g}a|v|^pdx >0$ by assumption~(\ref{a_assumptions}), which implies that $I(\lambda v) \le 0$ for sufficiently large $\lambda>0$, as needed.

Hence Theorem~\ref{thm:thm2.1CM} applies and yields the existence of a solution $u$ of (\ref{P_g}) satisfying (\ref{eq:K-ground-state-g}). The proof of Theorem~\ref{thm:main_existence-nonradial} is thus finished.

\appendix
\section{Proof of Theorem \ref{thm:thm2.1CM}}\label{sec:appendix}

Let $V=H^1_0(\Omega)\cap L^p(\Omega)$. Notice that $V$ equipped with the norm
\[
\|\cdot\|_V=\|\cdot\|_{H^1(\Omega)}+\|\cdot\|_{L^p(\Omega)}
\]
is a Banach space. 
The formal Euler-Lagrange functional \eqref{eq:J_def} is defined more precisely as $J:V\to (-\infty,+\infty]$ given by
\begin{equation}\label{eq:J_def2}
J(u)=\Psi(u)-\Phi(u),
\end{equation}
where, for every $u\in V$,
\[
\Psi(u):=\begin{cases}
\frac{\|u\|_{H^1(\Omega)}^2}{2} \ &\text{if }u\in K \\
+\infty & \text{otherwise}
\end{cases}
\qquad\quad
\Phi(u):=\frac{1}{p}\int_\Omega w|u|^p dx.
\]
Notice that $\Phi\in C^1(V;\R)$ and let, for $u,v\in V$,
\begin{equation}\label{eq:def_G}
G(u,v):=\Psi(u)-\Psi(v)- \Phi'(u)(u-v).
\end{equation}
Let us show that every Palais-Smale sequence, in the sense of \cite[Theorem 3.2]{S}, is bounded.
%
%
Let $(u_n)\subset K$ be a Palais-Smale sequence, i.e. 
\begin{equation}\label{eq:def_PS_sulkin}
J(u_n)\to c\in\R, \qquad
G(u_n,v)\leq \varepsilon_n \|u_n-v\|_V \ \ \forall v\in V,
\end{equation}
with $\varepsilon_n\to0$. Let
\begin{equation}\label{eq:alpha_beta_def}
\beta\in (1,p-1) \quad\text{ and }\quad \alpha=\frac{1}{p(\beta-1)}
\end{equation}
On the one hand, taking $v= \beta u_n$ in \eqref{eq:def_PS_sulkin}, we have
\begin{equation}\label{eq:PS_calc}
J(u_n)+\alpha G(u_n,\beta u_n) \leq c+1 +\alpha(\beta-1) \varepsilon_n \|u_n\|_V,
\end{equation}
where we used that $v\in K$, by the assumption that $K$ is a cone.
On the other hand, by \eqref{eq:J_def2}, \eqref{eq:def_G} and \eqref{eq:alpha_beta_def},
\begin{multline*}
J(u_n)+\alpha G(u_n,\beta u_n) =
\frac{\|u_n\|^2_{H^1(\Omega)}}{2} - \frac{1}{p}\int_\Omega w|u_n|^p dx 
+\alpha \frac{\|u_n\|^2_{H^1(\Omega)}}{2} \\
-\alpha \beta^2 \frac{\|u_n\|^2_{H^1(\Omega)}}{2}
+\alpha(\beta-1)\int_\Omega w|u_n|^p dx \\
= \frac{p-(\beta+1)}{2p} \|u_n\|^2_{H^1(\Omega)}.
\end{multline*}
By combining the last equality with \eqref{eq:PS_calc}, together with the fact that $\|\cdot\|_V$ and $\|\cdot\|_{H^1(\Omega)}$ are equivalent in $K$ by (i), we infer that $\|u_n\|_{H^1(\Omega)}$ is bounded.

The rest of the proof follows as in \cite{CowanMoameni2022}.

\section*{Acknowledgments}
\noindent 
A. Boscaggin was partially supported by the INdAM - GNAMPA Project 2023 ``Analisi qualitativa di problemi differenziali nonlineari''.
F. Colasuonno and B. Noris were partially supported by the INdAM - GNAMPA Projects 2022 ``Studi asintotici in problemi parabolici ed ellittici'' CUP\textunderscore E55F22000 270001 and 2023 ``Interplay between parabolic and elliptic PDEs'' CUP\textunderscore E53C22001 930001. 
B. Noris was partially supported by the MUR-PRIN project 2022R537CS ``NO$^3$'' granted by the European Union - Next Generation EU.


\def\cprime{$'$}

\end{document}